\documentclass[10pt,reqno]{amsart}
\usepackage{graphicx}
\usepackage{times,amsmath,cancel,stmaryrd,graphicx}
\usepackage{amsfonts,enumitem,amssymb,color}
\usepackage{lmodern}
\usepackage[colorlinks=true,linkcolor=cyan,citecolor=magenta]{hyperref}
\usepackage{cleveref,enumerate}
\usepackage{graphicx}
\usepackage{subcaption}
\usepackage{blindtext}
\usepackage{url}
\usepackage{multirow}
\usepackage{amsthm}
\usepackage{setspace}
\usepackage{float}
\usepackage{mathrsfs}
\usepackage{lineno}
\usepackage[title]{appendix}
\usepackage{xcolor}
\usepackage{array}
\usepackage{epstopdf}
\usepackage{epsfig}
\usepackage{textcomp}
\usepackage{manyfoot}
\usepackage{booktabs}
\usepackage{algorithm}
\usepackage{algorithmicx}
\usepackage{algpseudocode}
\usepackage{listings}
\usepackage{siunitx}
\usepackage{longtable}
\usepackage{threeparttable}
\usepackage{rotating}
\usepackage[numbers]{natbib}
\usepackage{caption}
\usepackage{orcidlink}
\newtheorem{thm}{Theorem}[section]

\newtheorem{rems}[thm]{Remarks}

\newtheorem{deff}[thm]{Definition}

\numberwithin{equation}{section}

\newcommand{\CC}{\mathbb{C}}

\newcommand{\RR}{\mathbb{R}}

\newcommand{\cO}{\mathcal{O}}

\newcommand{\bfA}{\mathbf{A}}
\newcommand{\bfB}{\mathbf{B}}
\newcommand{\bfC}{\mathbf{C}}

\newcommand{\bfE}{\mathbf{E}}

\newcommand{\bfI}{\mathbf{I}}

\newcommand{\bfK}{\mathbf{K}}

\newcommand{\bfM}{\mathbf{M}}

\newcommand{\bfU}{\mathbf{U}}
\newcommand{\bfV}{\mathbf{V}}

\newcommand{\bfX}{\mathbf{X}}

\newcommand{\bfb}{\mathbf{b}}
\newcommand{\bfc}{\mathbf{c}}

\newcommand{\bfl}{\mathbf{l}}

\newcommand{\bfp}{\mathbf{p}}
\newcommand{\bfq}{\mathbf{q}}
\newcommand{\bfr}{\mathbf{r}}
\newcommand{\bfs}{\mathbf{s}}

\newcommand{\bfv}{\mathbf{v}}

\newcommand{\gz}{\zeta}

\newcommand{\gl}{\lambda}

\title{A New Class of General Linear Method with Inherent Quadratic Stability for Solving Stiff Differential Systems}

\author{
  Sakshi Gautam\,\orcidlink{0009-0007-3010-0853}$^{\ast}$ \and
  Ram K. Pandey\,\orcidlink{0000-0002-8176-8729}$^{\dagger}$
}

\address{$^{\ast,\dagger}$Department of Mathematics and Statistics, 
Dr.\ Harisingh Gour Vishwavidyalaya (A Central University),
Sagar, Pin–470001, Madhya Pradesh, India.}

\email{$^{\ast}$sakshi94nikhil@gmail.com}
\email{$^{\dagger}$pandeywavelet@gmail.com}


\date{\today}
\begin{document}
\begin{abstract}
This article proposes a new class of general linear method (GLM) with $p=q$ and $r=s=p+1$. The construction of the present method is carried out using order conditions and error minimization subject to \emph{A}- stability constraints. The proposed time integration schemes are \emph{A}- and \emph{L}-stable GLMs equipped with inherent quadratic stability (IQS) criteria. We construct implicit GLMs of orders up to four with $p = q$ and $s = r$ along with the Nordsieck input vector assumption. Further, we test these schemes on three real-world problems: the van der Pol oscillator and two partial differential equations (Burgers' equation and the Gray-Scott model), and numerical results are presented. Computational results confirm that our proposed schemes are competitive with the existing GLMs and can be recognized as an alternative time integration scheme. We demonstrate the order of accuracy and convergence for the proposed schemes through observed order computation and error versus step size plots.
\end{abstract} 

\maketitle
\pagestyle{myheadings}
\markboth{\sc{Sakshi Gautam \& Ram K. Pandey}}{\sc{A new class of GLM with IQS for stiff systems}}
\section{Introduction}\label{sec1}
Many important physical and chemical process are described by partial differential equations (PDEs). Examples include transport phenomena, reaction-diffusion systems, and fluid dynamics. These PDEs, upon spatial discretization, result in large systems of ordinary differential equations (ODEs) \cite{Hundsdorfer2003}. However, these differential systems often exhibit complex phenomena that require a suitable time integrator to preserve the solution behavior. The system of time-dependent differential equations in the autonomous form is given by \cite{Butcher2016}
\begin{equation}\label{1.1}
\begin{cases}
    y^\prime=f(y(t)), \quad t\in[t_0,T],\\
    y(t_0)=y_0\in\RR^d,
\end{cases}
\end{equation}
where, $ f: \mathbb{R}^d \rightarrow \mathbb{R}^d,$ with the assumption that \eqref{1.1} is well-posed and has a unique solution.
In the present work, the partial differential equation governing physical phenomena such as advection-diffusion and reaction-diffusion has been considered. To apply time integration techniques to such a PDE, we need to reduce it to what a time integrator identifies, which is \eqref{1.1} using spatial discretization. But, in problems arising from real-world phenomena, this system is usually stiff, i.e., it requires a range of suitably small step sizes for explicit numerical solvers to obtain stable and convergent solutions, which in turn leads to large computational time and effort. To handle \eqref{1.1} when it exhibits stiffness, implicit solvers are typically employed. Therefore, choosing an appropriate time integrator becomes crucial \cite{Butcher2003}. General Linear Methods (GLMs) provide a flexible framework that includes Runge–Kutta (RK) and linear multistep methods (LMM) as special cases \citep{Butcher2016, Butcher1993}. Recent developments in GLMs have produced new classes that are not only accurate but also reliable for such problems \cite{Jaust2018}. 

The development of temporal integration schemes based on GLMs has already seen several advancements; GLM and its variants are used to solve a variety of different problems \cite{Ghahremani2025, Yu2025, Izzo2025}. For more on GLMs, see \cite{Jackiewicz2009, Cardone2011}, and the references cited therein. In recent years, Jackiewicz \cite{Jackiewicz2009} and Conte et al. \cite{Conte2010} built two-step RK and GLM versions with inherent quadratic stability (IQS). Cardone and co-authors \cite{Cardone2011, Cardone2012, Bras2013} studied IQS-based GLMs for non-stiff systems. Later, Bra\'s and his co-authors applied these ideas to stiff-problems \cite{Bras2011, Bras2014, Bras2014a}. More recently, efficient GLMs with quadratic stability (QS) functions have been constructed to lower the cost of stage computation where $s = p-1$ \cite{Bras2020}. However, the implicit GLMs with IQS for stiff problems for $s = p+1$ have not been constructed or analyzed.

This work aims to fill the gap in the literature for methods with $s=p+1$. Moreover, the $s=p+1$ design adds a stage in comparison to the $s = p$ structures in \cite{Bras2011}. However, the methods with $s=p$ are prone to having large error constants, but methods with $s=p+1$ has access to perhaps lower error constants \cite{Butcher2001}. This configuration with $r=s=p+1$ is very helpful in variable step size and variable order implementation as mentioned in \cite{Butcher2001}. The proposed methods remain competitive in accuracy, avoid order reduction, and demonstrate robustness across a range of test problems. We propose and examine GLM using IQS approaches that exhibit $A$-stability and $L$-stability, characterized by equal order and stage order ($p=q$) and with $r = s = p+1$. The choice of $r = p+1$ enables Nordsieck representation for the input/output vector. Constructing the IMEX equivalent of GLM with $p=q$ and $r=s=p+1$ with IQS as our future work is another justification behind the construction of these classes. We consider the construction and implementation of implicit classes of GLM with IQS. In addition, we present some example classes of the proposed schemes, along with computational results for three test problems that govern different physical phenomena.

The remainder of the article is structured as follows: Section \ref{sec3} presents our classes of implicit GLMs based on Nordsieck input vectors with IQS. Example classes of these methods for orders $p = 1, 2, 3$ and $4$ have been presented in Section \ref{sec4}. Further, Section \ref{sec5} demonstrates the applicability of the proposed classes to various modeled real-world problems, and the results are discussed in Section \ref{sec6}. Lastly, in Section \ref{sec7} we give the conclusion of the present work and future works in this direction.


\section{Construction of implicit GLM with IQS based on Nordsieck input vector}\label{sec3}  
For the initial value problem (IVP) given by \eqref{1.1}, the GLM class for one time step $t_n=t_{n-1}+h$ is given as follows \cite{Butcher2016, Jackiewicz2009}:
\begin{equation}\label{2.1.1}
\begin{aligned}
    Y_j^{[n]}=h\sum_{k=1}^{s}a_{jk}f(Y_k^{[n]})+ \sum_{k=1}^{r}u_{jk}y_k^{[n-1]},\quad j=1,2,\cdots,s,\\
    y_j^{[n]}=h\sum_{k=1}^{s}b_{jk}f(Y_k^{[n]})+ \sum_{k=1}^{r}v_{jk}y_k^{[n-1]},\quad j=1,2,\cdots,r, 
\end{aligned} 
\end{equation}
where, $\bfA=[a_{jk}]_{s\times s}$, $\bfU=[u_{jk}]_{s\times r}$,  $\bfB=[b_{jk}]_{r\times s}$ and  $\bfV=[v_{jk}]_{r\times r}$, are the coefficients of the method with $\bfc=[c_1,c_2,\cdots,c_s]^T$. The method $(\bfc,\bfA,\bfU,\bfB,\bfV)$ can be rewritten as
\begin{equation}\label{2.1.2}
\begin{aligned}
    Y^{[n]}=(\bfA\otimes \bfI) hf(Y^{[n]})+ (\bfU\otimes \bfI)y^{[n-1]},\\
    y^{[n]}=(\bfB\otimes \bfI) hf(Y^{[n]})+ (\bfV\otimes \bfI)y^{[n-1]},
\end{aligned} 
\end{equation}
where `$\otimes$' is the Kronecker product and $\bfI \in \RR^{d\times d}$. Here, the components $Y^{[n]}\in\RR^{sd}$ and $y^{[n]}\in\RR^{rd}$ are the internal and external stage approximations
\begin{equation}\label{inoutdef}
Y^{[n]}=  \left[
\begin{array}{c}
    Y_1^{[n]}\\
    Y_2^{[n]}\\
    \vdots\\
    Y_s^{[n]}
\end{array}
\right],\text{ } f(Y^{[n]})=  \left[
\begin{array}{c}
    f(Y_1^{[n]})\\
    f(Y_2^{[n]})\\
    \vdots\\
    f(Y_s^{[n]})
\end{array}
\right], \text{ and } y^{[n]}=  \left[
\begin{array}{c}
    y_1^{[n]}\\
    y_2^{[n]}\\
    \vdots\\
    y_r^{[n]}
\end{array}
\right].
\end{equation}
The matrices $\bfA$ and $\bfV$ for GLM \eqref{2.1.1} with ($r=s$) have the following structure 
\begin{equation}\label{IQSmatassume}
\bfA=
\begin{bmatrix}
    \gl & 0      & 0      & \cdots & 0 \\
    a_{21} & \gl & 0      & \cdots & 0 \\
    a_{31} & a_{32} & \gl & \cdots & 0 \\
    \vdots & \vdots & \vdots & \ddots & \vdots \\
    a_{r1} & a_{r2} & a_{r3} & \cdots & \gl
\end{bmatrix}, \quad \bfV=\begin{bmatrix}
    1      & v_{12} & v_{13} & \cdots & v_{1r} \\
    0      & 0      & v_{23} & \cdots & v_{2r} \\
    \vdots & \vdots & \vdots & \ddots & \vdots \\
    0 &0 &0 & \cdots & v_{r-1,r}  \\
    0      & 0      & 0      & \cdots & 0
\end{bmatrix}.
\end{equation}
Here, $\gl >0$ for the implicit classes \cite{Butcher1993, Butcher2001, Butcher1997} and $\bfV$ has this specific form with $v_{11}=1$, guaranteeing the zero stability and the power boundedness of the matrix $\bfV$ for the GLMs to be constructed \cite{Cardone2011, Bras2011}. \\

The internal and external stage approximations satisfy the following order assumptions \cite{Butcher2016, Jackiewicz2009, Cardone2011, Bras2020}.
\begin{itemize}
\item $Y^{[n]}$ is an approximation of stage order $q$, i.e.,\\\
\begin{equation}\label{2.1.3}
    Y^{[n]}= y(t_{n-1}+\bfc h)+\mathcal{O}(h^{q+1}),
\end{equation}
\item $y^{[n]}$ is an approximation of order $p$\\
\begin{equation}\label{2.1.4}
    y^{[n]}=\sum_{m=0}^{p}(\bfl_{m}\otimes \bfI) h^m y^{(m)}(t_{n})+\mathcal{O}(h^{p+1}).
    \end{equation}
  \item $y^{[n-1]}$ is also an approximation of order $p$\\
\begin{equation}\label{2.1.5}
    y^{[n-1]}=\sum_{m=0}^{p}(\bfl_{m}\otimes \bfI) h^m y^{(m)}(t_{n-1})+\mathcal{O}(h^{p+1}).  
\end{equation}
\end{itemize}
\begin{rems}
For the input vector in the Nordsieck form, the vectors $\bfl_m$ become the standard basis of  $\RR^r$, i.e., $\bfl_0=e_1,\bfl_1=e_2,\cdots,\bfl_p=e_r$ \cite{Bras2020}. 
\end{rems}
Here, we describe the construction of GLM with IQS, and this can be achieved using order conditions and the conditions obtained due to IQS imposition. The plan of the construction is divided under the following subheads. 
\subsection{Order conditions}\label{subsec3.1}
Assume that the GLM has input vector in the Nordsieck form, $y^{[n-1]} =  z(t_{n-1})+\mathcal{O}(h^{p+1})$  i.e., $y^{[n-1]}$ is an approximation of the Nordsieck vector of order $p$ with $y_j(t_{n-1})=h^{j-1}y^{(j-1)}(t_{n-1})+\mathcal{O}(h^{p+1})$ for $j=1,2,\ldots,r$. The Nordsieck vector of order $p=r-1$ is given as follows:
\begin{equation*}
z(t_{n-1})=\begin{bmatrix}
    y(t_{n-1})\\
    hy^\prime(t_{n-1})\\
    \vdots\\
    h^{r-1}y^{(r-1)}(t_{n-1})
\end{bmatrix}.
\end{equation*}
The order conditions required to construct the Nordsieck input vector-based GLM are given in the following theorem. 
\begin{thm}\label{The3}
\cite{Jackiewicz2009} A GLM $(\textbf{c},\textbf{A},\textbf{U},\textbf{B},\textbf{V})$ has order $p$ and stage order $q$ ($p=q=r-1$) with input vector in the Nordsieck form, if and only if
\begin{equation}\label{3.1.1}
    \begin{aligned}
        e^{\bfc z} & =z \bfA e^{\bfc z} +\bfU W+\mathcal{O}\left(z^{r}\right), \\
        e^{ z}  W & =z \bfB e^{\bfc z} +\bfV W+\mathcal{O}\left(z^{r}\right),
    \end{aligned}
\end{equation}
where $W=[1,z,\ldots,z^{r-1}]^T$. 
\end{thm}
\begin{proof}
The proof is trivial and is given in \cite{Butcher1993, Jackiewicz2009}.
\end{proof}

\subsection{Imposition of inherent quadratic stability and minimum error constant}\label{subsec3.2}
Consider the linear test problem as given below
\begin{equation}\label{testp}
y^{\prime}(t)=\gz y(t), \quad t \geq 0,\quad \gz \in \CC,
\end{equation}
for $ \Re(\gz)<0$ and $|\gz|$ is the stiffness parameter.\\
We apply the method \eqref{2.1.2} to the test problem \eqref{testp} and arrive at the vector recurrence relation for one step of the GLM
\begin{equation}\label{stabrec}
y^{[n]}=\bfM(\omega) y^{[n-1]}, \text{ where } \omega=h\gz.
\end{equation}
The stability matrix $\bfM(\omega)$ \cite{Cardone2011, Cardone2012} is defined by
\begin{equation}\label{stabmat}
\bfM(\omega)=\bfV+ \omega \bfB(\bfI-\omega \bfA)^{-1} \bfU,
\end{equation}
and define $\hat{p}(\eta,\omega)$, the stability function as the characteristic polynomial of the matrix $\bfM(\omega)$, i.e.,
\begin{equation}
\hat{p}(\eta,\omega)= \operatorname{det}(\eta \bfI-\bfM(\omega)).
\end{equation}
As $\hat{p}(\eta,\omega)$ is a polynomial in $\eta$ with rational coefficients in $\omega$, so for simplification, we multiply $\hat{p}(\eta,\omega)$ by the factor $(1-\lambda \omega)^r$. The resulting stability function is denoted by $p(\eta, \omega)$ and is given by
\begin{equation}\label{stabfun}
p(\eta, \omega)=(1-\lambda \omega)^r \hat{p}(\eta,\omega),
\end{equation}
where, $\eta \in \mathbb{C}$.
In this article, we search for numerical methods that assume a stability function of the form given as
\begin{equation}\label{2.2.4}
p(\eta,\omega)=\eta^{r-2}((1-\lambda \omega)^r\eta^2-p_1(\omega)\eta+p_0(\omega)),
\end{equation}
where, $p_1(\omega)$ and $p_0(\omega)$ are polynomials with degree $r$. If the characteristic function of the underlying stability matrix $\bfM(\omega)$ has the form \eqref{2.2.4}, then the GLM $(\bfc,\bfA,\bfU,\bfB,\bfV)$ has a quadratic stability function. For a more detailed description of the stability analysis, see \cite{Butcher2016, Jackiewicz2009, Huang2005, Wright2002}.

To analyze GLM with IQS, first define a special matrix equivalence relation  $\textbf{E}\equiv\textbf{F}$ if $\textbf{E}=\textbf{F}$ except for the first two rows. To ensure the quadratic stability of GLM $(\textbf{c},\textbf{A},\textbf{U},\textbf{B},\textbf{V})$, we provide the IQS conditions in the following definition.
\begin{deff}\label{def1}
\cite{Cardone2011, Conte2010, Bras2011, Bras2014, Bras2012} The GLM $(\textbf{c},\textbf{A},\textbf{U},\textbf{B},\textbf{V})$ with $r=p+1$, possess inherent quadratic stability if there exist a matrix $\textbf{X}$ such that:
\begin{equation}\label{IQScondition}
    \textbf{BA}\equiv\textbf{XB} \quad\text{and}\quad \textbf{BU}\equiv\textbf{XV}-\textbf{VX}. 
\end{equation}
\end{deff}

This matrix equivalence relation for the so-called ``inherent quadratic stability (IQS)" condition defined in \eqref{IQScondition} is a sufficient condition for obtaining the stability polynomial in the form given by \eqref{2.2.4}.

We use the order conditions as defined in Theorem \ref{The3} and determine the coefficients for the matrix $\bfU$ and for the $r-1$ columns of $\bfB$. Now, for determining the rest of the parameters, we impose the algebraic conditions of IQS as defined in \eqref{IQScondition} with the assumption that the matrix $\bfX$ has the following form \cite{Cardone2012}
\begin{equation}\label{3.2.3}
\bfX = \begin{bmatrix}
    x_{11} & x_{12} & \cdots & x_{1,r-1} & x_{1r} \\
    x_{21} & x_{22} & \cdots & x_{2,r-1} & x_{2r} \\
    0      & 1      & \cdots & 0         & x_{3r} \\
    \vdots & \vdots & \ddots & \vdots    & \vdots \\
    0      & 0      & \cdots & 1         & x_{rr} \\
\end{bmatrix}.
\end{equation}
Solving the IQS conditions \eqref{IQScondition} for the $2(r-2)$ rows, we arrive at the desired stability function for GLM with IQS. The following theorem verifies the assertion.
\begin{thm}
 If the GLM $(\bfc,\bfA,\bfU,\bfB,\bfV)$ of order and stage order $p$ with Nordsieck input vector satisfies the order conditions \eqref{3.1.1} and IQS conditions \eqref{IQScondition} for the structure of $\bfX$ (given in \eqref{3.2.3}), then the stability polynomial takes the following form
\begin{equation}\label{3.2.4}
    p(\eta,\omega)=\eta^{r-2}((1-\gl \omega)^r\eta^2-p_1(\omega)\eta+p_0(\omega)),
\end{equation}
where,
\begin{equation}\label{3.2.5}
    \begin{aligned}
        &p_1(\omega)=1+\omega p_{11}+\omega^2p_{12}+\cdots+\omega^rp_{1r},\\
        &p_0(\omega)=\omega p_{01}+\omega^2p_{02}+\cdots+\omega^rp_{0r}.\\
    \end{aligned}
\end{equation}
\end{thm}
\begin{proof}
The proof is done along the lines of \cite{Cardone2011, Bras2011}. We start with the method of order $p$ satisfying \eqref{3.1.1}. Consider its stability matrix $\bfM(\omega)$ given in \eqref{stabmat}, and right multiply it by $(\bfI-\omega\bfX)$.
Assuming $(\bfI-\omega\bfX)$ is non singular and that the matrix equivalence \eqref{IQScondition} holds, \eqref{stabmat} becomes
\begin{equation*}
    \begin{aligned}
        &(\bfI-\omega\bfX)\bfM(\omega)=(\bfI-\omega\bfX)\bfV+\omega (\bfI-\omega\bfX)\bfB(\bfI-\omega \bfA)^{-1} \bfU\\
        &\hspace{2.35cm}=(\bfI-\omega\bfX)\bfV+(\bfB-\omega\bfX\bfB)(\bfI-\omega \bfA)^{-1}\bfU \\
        &\hspace{2.35cm}\equiv (\bfI-\omega\bfX)\bfV+(\bfB-\omega\bfB\bfA)(\bfI-\omega \bfA)^{-1}\bfU \\
        &\hspace{2.35cm}\equiv (\bfI-\omega\bfX)\bfV+\bfB(\bfI-\omega\bfA)(\bfI-\omega \bfA)^{-1}\bfU \\
        &\hspace{2.35cm}\equiv \bfV-(\omega\bfX\bfV-\bfB\bfU) \\
        &\hspace{2.35cm}\equiv \bfV-\omega\bfV\bfX. \\
    \end{aligned}
\end{equation*}  
Therefore, $(\bfI-\omega\bfX)\bfM(\omega)(\bfI-\omega\bfX)^{-1} \equiv  \bfV$. Let 
$(\bfI-\omega\bfX)\bfM(\omega)(\bfI-\omega\bfX)^{-1}=\tilde{\bfM}(\omega)$ defined using block matrices 
$\tilde{\bfM}_{11}(\omega)\in \RR^{2\times 2}$, 
$\tilde{\bfM}_{12}(\omega)\in \RR^{2\times (r-2)}$, 
$\tilde{\bfM}_{22}(\omega)\in \RR^{(r-2)\times (r-2)}$, 
and $\bf0 \in \RR^{(r-2)\times 2}$ i.e.,
\begin{equation*}
\tilde{\bfM}(\omega)=	\begin{bmatrix}
    \tilde{\bfM}_{11}(\omega) & \tilde{\bfM}_{12}(\omega)\\
    \bf0 & \tilde{\bfM}_{22}(\omega)
\end{bmatrix}.
\end{equation*}
Also define $\bfV$ in the same block format with 
$\tilde{\bfV}_{11}\in \RR^{2\times 2}$, 
$\tilde{\bfV}_{12}\in \RR^{2\times (r-2)}$, 
and $\tilde{\bfV}_{22}\in \RR^{(r-2)\times (r-2)}$, we have 
\begin{equation*}
\bfV=\begin{bmatrix}
    \tilde{\bfV}_{11} & \tilde{\bfV}_{12}\\
    \bf0 & \tilde{\bfV}_{22}
\end{bmatrix}.
\end{equation*}
Since $\tilde{\bfM}(\omega) \equiv \bfV$ except for the first two rows, we have
$\tilde{\bfM}_{22}(\omega)=\tilde{\bfV}_{22}$, 
and therefore $\tilde{\bfM}_{22}(\omega)$ is a strictly upper triangular matrix with eigen values $0$ of multiplicity $r-2$ and the stability function of the matrix $\bfM(\omega)$ becomes
\begin{equation}\label{reqstab}
\det(\eta \bfI-\bfM(\omega))=\eta^{r-2}\det(\eta \tilde{\bfI}-\tilde{\bfM}_{11}(\omega)),
\end{equation}
where, $\tilde{\bfI}\in \RR^{2\times 2}$ and \eqref{reqstab} gives the required stability function $p(\eta,\omega)$ of the form \eqref{3.2.4}.
\end{proof}
We search for implicit methods that are $A$-stable as well as $L$-stable, and possibly have a small error constant. For the stability function $p(\eta,\omega)$ given by \eqref{3.2.4} to be $L-$stable, the coefficients $p_{1r}$, and $p_{0r}$ in \eqref{3.2.5} must satisfy the following:    
\begin{equation}\label{3.2.6}
\begin{aligned}
    \lim_{\omega\to \infty}\frac{p_1(\omega)}{(1-\gl \omega)^r}=0,\quad \lim_{\omega\to \infty}\frac{p_0(\omega)}{(1-\gl \omega)^r}=0,
\end{aligned}
\end{equation}
which is equivalent to $p_{1r}=0$ and $p_{0r}=0$.

Our strategy in this paper is first to employ the $L$-stability criteria given by \eqref{3.2.6} to reduce the number of parameters to compute. 
Further, we define the local error at time $t_n$, which will then be used as an objective in the minimization problem for determining the remaining method coefficients.

Assuming $y_1^{[n-1]}= y(t_{n-1})+\cO(h^{p+1})$, and $y_k^{[n-1]}= h^{k-1}y^{(k-1)}(t_{n-1})+\cO(h^{p+1})$ for $k=2,3,\cdots,r$, the error propagation in one step of the GLM scheme is given as follows \cite{Cardone2011, Cardone2012}:
\begin{equation}\label{erroreq}
y(t_n)-y_1^{[n]}=\left(\frac{1}{(p+1)!}-\frac{\bfb^T \bfc^p}{p!}+\bfv^T\beta \right)h^{p+1}y^{(p+1)}(t_n)+\cO(h^{p+2})
\end{equation}
where, 
\begin{equation*}
\beta=(I-\tilde{V})^{-1}\left(\left[\frac{1}{p!},\frac{1}{(p-1)!},\cdots,1\right]^T-\frac{\tilde{B} \bfc^p}{p!}\right). 
\end{equation*}
Here, the matrices $\bfb^T$, $\bfv^T$, $\tilde{B}$ and $\tilde{V}$ are submatrices of the matrix $\bfB$ and $\bfV$ and are defined as follows \cite{Jackiewicz2009, Cardone2011, Cardone2012}:
\begin{equation*}
\bfB  =\left[
\begin{array}{c}
\bfb^T \\
\tilde{B}
\end{array}
\right], \text { and } \bfV =\left[\begin{array}{cc}
1 & \bfv^T \\
0 & \tilde{V}
\end{array}
\right].
\end{equation*} 
To ensure $A-$stability, we impose the Schur criterion on the eigenvalues of the stability matrix $\bfM(\omega)$, requiring that all eigenvalues lie inside or on 
the unit disk for $\Re(\omega) \leq 0$. This yields a system of nonlinear constraints in terms of the free parameters of $\bfV$ and $\gl$.
The error constant is the leading term in \eqref{erroreq} defined as
\begin{equation}\label{eq:error_constant}
E_{r} = 
\left| 
    \frac{1}{(p+1)!} 
    - \frac{\bfb^{T}\bfc^{p}}{p!} 
    + \bfv^{T}\beta 
\right|.
\end{equation}
After applying the Schur criteria, let $\mathcal{F}_{\text{Schur}}$ be the feasible region for the $A-$stability, which contains the possible range of values for the free parameters of the method coefficients that can be optimally solved by minimizing the error constant \eqref{eq:error_constant}
\begin{equation}\label{eq:optimization}
\begin{aligned}
&\min_{\text{free parameters}} \, E_{r}\\
&\text{subject to} \quad \mathcal{F}_{\text{Schur}}.
\end{aligned}
\end{equation}
We now present the example classes of constructed implicit GLMs of order up to $p=4$ in the following section.

\begin{algorithm}[h]
\caption{Workflow for construction of an Implicit GLM with IQS and $r=s=p+1$}
\label{alg:GLMcompact}
\begin{algorithmic}[1]
\State \textbf{Input} $p$ (order of method), $s=p+1$ (no. of internal stages), and $r=p+1$ (no. of external stages).
\State \textbf{Define symbolic matrices:}
Initialize the coefficient matrices $\bfA \in \RR^{s\times s}$, $\bfU\in \RR^{s\times r}$, 
        $\bfB\in \RR^{r\times s}$, $\bfV\in \RR^{r\times r}$, and abscissae $\bfc\in\RR^{s}$ with $\bfA$ and $\bfV$ as given in \eqref{IQSmatassume}.

\State \textbf{Choose abscissae:}  
      Define $\bfc=[c_1,\dots,c_{s}]^T$ with
      \[
      c_i=\frac{i-1}{\,s-1\,}, \qquad i=1,\dots,s.
      \]
      Construct the Vandermonde matrix of the abscissae $\bfc$
      \[
      \bfC_r = 
      \begin{bmatrix}
      \mathbf{1} & \bfc & \bfc^2/2! & \dots & \bfc^{\,r-1}/(r-1)!
      \end{bmatrix}
      \in \RR^{s\times r}.
      \]

\State \textbf{Define matrices:}  
      Introduce the upper shift matrix $\bfK_r\in\RR^{r\times r}$,  
      and
      \[
      \bfE_r = \exp(\bfK_r)\in\RR^{r\times r}.
      \]

\State \textbf{Solve order conditions:}  
    Solve matrix equations
      \[
      \bfC_r = \bfA\,\bfC_r\,\bfK_r + \bfU,  \qquad
      \bfE_r = \bfB\,\bfC_r\,\bfK_r + \bfV,
      \]
      to determine $\bfU$ and the first $(r-1)$ columns of $\bfB$. 

\State \textbf{Applying IQS constraints:}  
     Solve $(\bfB \bfA-\bfX \bfB)_{(3:s,:)}= \bf0$ and $\big(\bfB \bfU-(\bfX\bfV-\bfV\bfX)\big)_{(3:r,:)}= \bf0,$
      for all remaining free parameters in $\bfB$ and $\bfV$. Here $\bfX_{(3:s,:)}=\bf0$ means $\bfX$ is zero except for first $2$ rows. 

\State \textbf{Enforce L-stability:}  
      Form the stability function \eqref{3.2.4} and impose the L-stability conditions by setting :
      \[
      p_{1r}=0, \qquad p_{0r}=0,
      \]
      as given in \eqref{3.2.5}.

\State \textbf{\emph{A}-Stability via Schur criterion:}  
      Apply the Schur stability criteria to obtain the admissible range of the remaining free parameter(s). Denote this admissible set by $\mathcal{F}_{\mathrm{Schur}}$.

\State \textbf{Minimize the error constant.}  
      Solve the minimization problem   
      \[
      \min_{\text{free parameters}}
      E_r=\left|
      \frac{1}{(p+1)!}
      - \frac{\bfb^T \bfc^{\,p}}{p!}
      + \bfv^T\beta
      \right|,
      \qquad \text{subject to } \mathcal{F}_{\mathrm{Schur}}.
      \]

\State \textbf{Return} the coefficient matrices  
      $\bfA$, $\bfU$, $\bfB$, $\bfV$, and the abscissa $\bfc$.

\end{algorithmic}
\end{algorithm}

\section{Example Methods}\label{sec4}
We provide four example methods of the numerical integrators constructed using the order conditions defined in \eqref{3.1.1} and the stability imposition presented in Section \ref{subsec3.2}. The proposed classes of implicit GLMs with IQS are $A-$ and $L-$ stable. The workflow of step-by-step procedure for the construction of example methods is demonstrated in Algorithm \ref{alg:GLMcompact}. Following this procedure, we obtain parameter-dependent families of GLMs with IQS, referred to as GLMQS. The coefficient matrices for GLMQS methods of orders up to $p=4$ are provided in the subsequent subsections.
\subsection{GLMQS-1 with $\bfp=\bfq=\bf{1}$ and $\bfr=\bfs=\bf{2}$}\label{subsec4.1} 
To construct the class of implicit GLMs of order and stage order $1$, we use the Nordsieck input vector of order $1$, i.e., $y^{[n-1]}=z(t_{n-1})+\mathcal{O}(h^2)$ which is given by\\
\begin{equation}\label{4.1.1}
z(t_{n-1})=\begin{bmatrix}
    y(t_{n-1})\\
    hy^\prime(t_{n-1})
\end{bmatrix},\text{ and }Y^{[n]}=\begin{bmatrix}
    Y_1^{[n]}\\
    Y_2^{[n]}
\end{bmatrix},
\end{equation}
and, the collocation abscissa $\bfc=[0,1]^T$. This class of methods already has a quadratic stability function. After solving the order conditions, we first impose the $L-$stability criteria \eqref{3.2.6} and obtain all coefficients in terms of the parameters $v_{12}$ and $\gl$. 
The roots of the stability function are given as follows:
\begin{equation*}
\begin{small}
\begin{aligned}
\eta_1 = 
\frac{- 3\gl \omega - v_{12} \omega + \omega + 1
    -\sqrt{
        (3\gl \omega + v_{12} \omega - \omega - 1)^2
        - 4\omega(-\gl - v_{12})(\gl^2 \omega^2 - 2\gl \omega + 1)
    }}{
    2(\gl^2 \omega^2 - 2\gl \omega + 1)
},
\\
\eta_2 =
\frac{- 3\gl \omega - v_{12} \omega + \omega + 1+
    \sqrt{
        (3\gl \omega + v_{12} \omega - \omega - 1)^2
        - 4\omega(-\gl - v_{12})(\gl^2 \omega^2 - 2\gl \omega + 1)
    }}{
    2(\gl^2 \omega^2 - 2\gl \omega + 1)
}  .    
\end{aligned}
\end{small}
\end{equation*}
We then solve the minimization problem subject to the constraints of $A-$ stability criteria to obtain the values of the free parameters $v_{12}$ and $\gl$. The GLM coefficients can then be expressed in the Butcher tableau as follows:
\begin{equation}
M_1=\left[\begin{array}{cc|cc}
    \gl & 0 &1 & u_{12}\\
    a_{21} & \gl &1 &u_{22}\\
    \hline
    b_{11} &b_{12} & 1 & v_{12}\\
    b_{21} &b_{22} & 0 & 0
\end{array}\right].
\end{equation}
We arrive at the class of GLMs with method coefficient $(\bfc,\bfA,\bfU,\bfB,\bfV)$ as given below: 
        \[
\small
\bfA = 
\begin{bmatrix}
0.4779022865816724 & 0 \\
1 & 0.4779022865816724
\end{bmatrix},
\]
\[
\small
\bfU = 
\begin{bmatrix}
1 & -0.4779022865816724 \\
1 & -0.4779022865816724
\end{bmatrix},
\]
\[
\small
\bfB = 
\begin{bmatrix}
0.9999999999996634 & 0.47790228658136436 \\
0.5220977134183276 & 0.4779022865816724
\end{bmatrix},
\]
\[
\small
\bfV = 
\begin{bmatrix}
1 & -0.4779022865810278 \\
0 & 0
\end{bmatrix}.
\quad
\]
The error constant for this example method is $E = 0.22741$, and the stability function having two non-zero eigenvalues is given as follows:
\begin{equation}
\begin{aligned}
&p(\eta,\omega)=  \eta^2 (-1 + 0.477902 \omega)^2+\eta (-1. - 0.0441954 \omega)\\  
&\quad \quad \quad \quad   - (6.44595\times10^{-13} + 1.11022\times10^{-16} \omega) \omega.
\end{aligned}
\end{equation}

\subsection{GLMQS-2 with $\bfp=\bfq=\bf{2}$ and $\bfr=\bfs=\bf{3}$}\label{subsec4.2}
To obtain the example method of order $2$, first solve the order conditions, and then solve the IQS stability criteria \eqref{IQScondition}. In addition, the $L-$stability criteria \eqref{3.2.6} is solved to reduce the coefficients of the method in terms of the free parameters $v_{13}$ and $\gl$. The stability function in terms of the free parameters has the form $p(\eta,\omega)=\eta((1-\gl \omega)^3\eta^2-p_1(\omega)\eta+p_0(\omega))$ with 
\begin{equation}
\begin{aligned}
   & p_1(\omega)=\frac{1}{2}  (2 + (3 - 10 \gl + 2 \gl^2) \omega + (1 - 7 \gl + 12 \gl^2 - 4 \gl^3 - 2 v_{13}) \omega^2),\\
   & p_0(\omega)=\frac{1}{2} \omega (1 + \omega - 4 \gl^3 \omega - 2 v_{13} \omega - 2 \gl^2 (1 + 4 \omega) - \gl (4 + 5 \omega)).
\end{aligned}
\end{equation}
Further to evaluate the values of the free parameters in $p(\eta,\omega)$, the minimization problem \eqref{eq:optimization} is solved under the constraints obtained by Schur stability criteria. The implicit GLM with IQS, of order $p=q=2$, and $r=s=3$ with the Nordsieck input vector and predetermined collocation abscissa $\bfc=[0,1/2,1]^T$ is given as follows:
    \[
    \small
    \bfA =
    \begin{bmatrix}
    0.4127594486653355 & 0 & 0 \\
    0.5 & 0.4127594486653355 & 0 \\
    0.5 & 0.5 & 0.4127594486653355
    \end{bmatrix},
    \]
    
    \[
    \small
    \bfU = 
    \begin{bmatrix}
    1 & -0.4127594486653355 & 0 \\
    1 & -0.4127594486653355 & 0.04362027566733226 \\
    1 & -0.4127594486653354 & -0.16275944866533548
    \end{bmatrix},
    \]
    
    \[
    \small
    \bfB = 
    \begin{bmatrix}
    0.08251725509138857 & 1.1935839192127649 & -0.10573081184164185 \\
    -0.825518897330671 & 1.8255188973306709 & 0 \\
    -2 & 2 & 0
    \end{bmatrix},
    \]
    
    \[
    \small
    \bfV = 
    \begin{bmatrix}
    1 & -0.17037036246251172 & 0.00893885223525935 \\
    0 & 0 & 0.08724055133466452 \\
    0 & 0 & 0
    \end{bmatrix}.
    \]
 Here, the error constant, $E=0.0195824$, with the stability function
\begin{equation}
\begin{aligned}
& p(\eta,\omega)=\eta(\eta^2 (-1 + 0.412759 \omega)^3 - \eta (-1 + 0.393427 \omega + 0.0720187 \omega^2)\\  
&\quad \quad \quad \quad +(0.155149 - 2.21936\times10^{-8} \omega) \omega).
\end{aligned}
\end{equation}	
\subsection{GLMQS-3 with $\bfp=\bfq=\bf{3}$ and $\bfr=\bfs=\bf{4}$}\label{subsec4.3} 
Following a similar construction as in order $2$ and a simplifying assumption $v_{13}=v_{12}$, we obtain the method coefficient in terms of the free parameters $v_{12}$ and $\gl$. The stability function $p(\eta,\omega)=\eta^2((1-\gl \omega)^3\eta^2-p_1(\omega)\eta+p_0(\omega))$ with
\begin{equation}
\begin{aligned}
  & p_1(\omega)=  (1 - \frac{1}{3} (-4 + 21 \gl + 3 v_{12}) \omega + 
\frac{1}{27} (17 + 324 \gl^2 - 27 v_{12} + 9 \gl (-16 + 9 v_{12})) \omega^2 +\\&\quad\quad\quad\quad \frac{1}{54}(-7 - 378 \gl^3 +\gl (20 - 54 v_{12}) - 108 \gl^2 (-2 + v_{12}) + 81 v_{12}) \omega^3),\\
& p_0(\omega)=\left(\frac{1}{3} - 3\gl - v_{12}\right) \omega 
 + \left(\frac{25}{54} - \frac{13\gl}{3} + 6\gl^2 - 2v_{12} + 3\gl v_{12}\right) \omega^2 
 +\\
 &\quad\quad\quad\quad  \left(-\frac{25\gl}{54} + 4\gl^2 - 3\gl^3 + 2\gl v_{12} - 2\gl^2v_{12}\right) \omega^3.
\end{aligned}
\end{equation}
We present the class of implicit GLMs of order and stage order $3$ with the Nordsieck input vector and the collocation abscissa $\bfc=[0,1/3,2/3,1]^T$. The GLM coefficient matrices $(\bfA,\bfU,\bfB,\bfV)$, with the error constant $E=7.729463\times10^{-10}$ are given below:
\begin{equation*}
\begin{small}
\bfA = 
\begin{bmatrix}
1.3070643469 & 0 & 0 & 0 \\
0.3333333333 & 1.3070643469 & 0 & 0 \\
0.3333333333 & 0.3333333333 & 1.3070643469 & 0 \\
0.3333333333 & 0.3333333333 & 0.3333333333 & 1.3070643469
\end{bmatrix},
\end{small}
\end{equation*}

\begin{equation*}
\begin{small}
\bfU = 
\begin{bmatrix}
1 & -1.3070643469 & 0 & 0 \\
1 & -1.3070643469 & -0.3801325601 & -0.0664418464 \\
1 & -1.3070643469 & -0.7602651202 & -0.2595945462 \\
1 & -1.3070643469 & -1.1403976803 & -0.5794580994
\end{bmatrix},
\end{small}
\end{equation*}

\begin{equation*}
\begin{small}
\bfB = 
\begin{bmatrix}
-0.8343558447 & 2.1518400434 & -0.3006125529 & 0.9548594035 \\
5.9455090739 & -19.7334042294 & 14.7878951555 & 0 \\
14.7635791223 & -32.5271582445 & 17.7635791223 & 0 \\
9 & -18 & 9 & 0
\end{bmatrix},
\end{small}
\end{equation*}

\begin{equation*}
\begin{small}
\bfV = 
\begin{bmatrix}
1 & -0.9717310493 & -0.9717310493 & -0.3635069146 \\
0 & 0 & -2.2807953605 & -1.6898986885 \\
0 & 0 & 0 & -1.1403976803 \\
0 & 0 & 0 & 0
\end{bmatrix}.
\end{small}
\end{equation*}
The resulting stability function of this class of methods is given as 
\begin{equation}
\begin{aligned}
&p(\eta,\omega)=\eta^2(\eta^2 (-1+1.30706 \omega)^4  +\eta\left(5.31018 \omega^3-11.321 \omega^2+6.84439 \omega-1\right)\\ 
&\quad \quad \quad \quad +\omega\left(0.309527 \omega^2+3.18264 \omega-2.61613\right)).
\end{aligned}
\end{equation}

\subsection{GLMQS-4 with $\bfp=\bfq=\bf{4}$ and $\bfr=\bfs=\bf{5}$}\label{subsec4.4} 
In this example method, we solve the minimization problem with constraints obtained using Schur stability criteria for the free parameters $\gl, v_{12}, v_{13}$ and $v_{14}$. The stability polynomial for this class of method is $p(\eta,\omega)=\eta^2((1-\gl \omega)^3\eta^2-p_1(\omega)\eta+p_0(\omega))$, we omit writing the explicit coefficients $p_{1}(\omega)$ and $p_{0}(\omega)$ due to its complex expression. The implicit GLM with IQS of order $p=q=4$, and $r=s=4$ with predetermined collocation abscissa $\bfc=[0,1/4,2/4,3/4,1]^T$ is given as follows:
\begin{equation*}
\begin{footnotesize}
\bfA = 
\begin{bmatrix}
1.14488604 & 0 & 0 & 0 & 0 \\
0.25 & 1.14488604 & 0 & 0 & 0 \\
0.25 & 0.25 & 1.14488604 & 0 & 0 \\
0.25 & 0.25 & 0.25 & 1.14488604 & 0 \\
0.25 & 0.25 & 0.25 & 0.25 & 1.14488604
\end{bmatrix},
\end{footnotesize}
\end{equation*}

\begin{equation*}
\begin{footnotesize}
\bfU = 
\begin{bmatrix}
1 & -1.14488604 & 0 & 0 & 0 \\
1 & -1.14488604 & -0.25497151 & -0.03317352 & -0.00281871 \\
1 & -1.14488604 & -0.50994302 & -0.13008992 & -0.02189867 \\
1 & -1.14488604 & -0.76491453 & -0.29074920 & -0.07317558 \\
1 & -1.14488604 & -1.01988604 & -0.51515135 & -0.17258517
\end{bmatrix},
\end{footnotesize}
\end{equation*}

\begin{equation*}
\begin{footnotesize}
\bfB = 
\begin{bmatrix}
43.96171205 & -203.73777224 & 341.62582482 & -248.83459442 & 69.31103311 \\
-57.45201209 & 215.29165614 & -271.46590848 & 114.62626443 & 0 \\
-33.44194715 & 138.96219468 & -181.59854791 & 76.07830038 & 0 \\
-97.27270647 & 307.81811940 & -323.81811940 & 113.27270647 & 0 \\
-64 & 192 & -192 & 64 & 0
\end{bmatrix},
\end{footnotesize}
\end{equation*}

\begin{equation*}
\begin{footnotesize}
\bfV = 
\begin{bmatrix}
1 & -1.32620332 & -2.06355665 & -0.84054293 & -0.60062733 \\
0 & 0 & -3.05965812 & -4.53326256 & -2.79810815 \\
0 & 0 & 0 & -2.03977208 & -1.42783313 \\
0 & 0 & 0 & 0 & -1.01988604 \\
0 & 0 & 0 & 0 & 0
\end{bmatrix}.
\end{footnotesize}
\end{equation*}
The error constant for this method is $E=2.25574\times10^{-8}$ and the corresponding stability function is given as follows:
\begin{small}
\begin{equation}
\begin{aligned}
&p(\eta,\omega)=\eta^3(\eta^2 (0.839345 \omega-1)^5-\eta \left(-1+3.38067 \omega+2.49947 \omega^2-0.273379 \omega^3+0.459926 \omega^4\right)\\
& \quad \quad \quad \quad \quad +\omega (-0.725147 \omega^3-0.132481 \omega^2+1.03275 \omega+0.183942)).
\end{aligned}
\end{equation}
\end{small}

\section{Numerical Experiments}\label{sec5}
In this section, we present numerical illustrations for the GLM solvers constructed in Section \ref{sec4} using three well-known test problems that model various real-world phenomena. First, we consider the van der Pol problem, a nonlinear oscillator equation used in fields such as electronics, physics, biology, and acoustics. One of its primary applications is the simulation of the human heartbeat, particularly when the damping coefficient is large. We then perform numerical simulations for Burgers' equation, a classical PDE that arises in gas dynamics, nonlinear acoustics, fluid mechanics, traffic flow, and several other areas. Finally, we solve the Gray–Scott model, a reaction–diffusion system that describes how chemical species move, react, and spread. The Gray–Scott model is relevant in many physical and biological contexts and is used in applications including biology, epidemiology, and medicine \cite{Karaagac2021}. The effective order of accuracy of the proposed methods is estimated using the formula:
\begin{equation}\label{OrderEst.}
p = \frac{\log \left(\frac{\left\|e_{N_{1}}(T)\right\|}{\left\|e_{N_{2}}(T)\right\|}\right)}{\log \frac{N_{2}}{N_{1}}},
\end{equation}
where $ \|e_{N_{1}}(T)\| $, $ \|e_{N_{2}}(T)\| $ denote the error norms at time $ T $ when using two consecutive time steps $ N_1, N_2 $ respectively. The time steps  $ N_1, N_2 $ are given by  
\begin{equation}
h_1 = \frac{T - t_0}{N_1},\quad h_2 = \frac{T - t_0}{N_2},
\end{equation}
where $ h_1, h_2 $ are the step sizes. All the order calculations for Tables \ref{table1}, \ref{table2} and \ref{table3} are done using \eqref{OrderEst.}.
To assess the accuracy and efficiency of the proposed solvers, we compute error norms from numerical simulations while refining the time grid. The error associated with the step size $h$ is denoted by $ e_{h} $. For any general method of order $p$, the global error is expected to behave as follows:  
\begin{equation}\label{graphorderest}
\begin{aligned}
    & \|e_{h}\| =C_1h^p+\mathcal{O}(h^{p+1}),\\
    & \|e_{h}\| \approx C_1h^p,\\
\end{aligned}
\end{equation}
for some constant $ C_1 > 0 $. Taking natural logarithm on both sides, we obtain $\log \|e_{h}\|\approx \log (C_1h^p)=\log(C_1)+p \log{h}$, this represents a linear relationship between $\log \|e_{h}\|$ and $\log h$ with slope $p$. The order shown in the convergence plots of the problems taken in this section is obtained using the relation \eqref{graphorderest}. Also, as GLMs are not self-starting methods, for starting the method, we use the starting procedure given by \cite{Califano2017} based on the approximations to $y$ at equally spaced arguments.
\subsection{Van der Pol problem}
The van der Pol problem \cite{Izzo2025, Hairer1996} is the system of two first-order differential equations obtained from the second-order oscillator differential equation given below
\begin{equation}\label{vdpol}
    \begin{aligned}
        y^{\prime} &= z, \\
        z^{\prime} &= \frac{(1-y^2)z - y}{\epsilon}.
    \end{aligned}
\end{equation}
Here, the initial conditions are taken as $y(0) = 2$ and $z(0) = -\frac{2}{3} + \frac{10}{81}\epsilon - \frac{292}{2187}\epsilon^2 - \frac{1814}{19683}\epsilon^3 + \cO(\epsilon^4)$. We have solved this problem for $\epsilon=10^{-6}$, for which the problem becomes highly stiff in nature. The interval of integration $t \in [0,0.5]$, with end point $T=0.5$.
\begin{table}[h]
\caption{End‑point $L_2$~‑error norm $\|e_{N_i}\|$ and estimated order $p$ for the van der Pol problem}
\label{table1}
\begin{small}
\begin{tabular*}{\textwidth}{@{\extracolsep\fill}lcccccccc}
\toprule
& \multicolumn{4}{c}{$p = 1$} & \multicolumn{4}{c}{$p = 2$} \\
\cmidrule(lr){2-5} \cmidrule(lr){6-9}
&\multicolumn{2}{c}{GLMQS} & \multicolumn{2}{c}{Bra\'s \cite{Bras2011}} & \multicolumn{2}{c}{GLMQS} & \multicolumn{2}{c}{Bra\'s \cite{Bras2011}}\\
\cmidrule(lr){2-3} \cmidrule(lr){4-5} \cmidrule(lr){6-7}\cmidrule(lr){8-9}
$N_i$ & $\|e_{N_i}\|$ & $p$ & $\|e_{N_i}\|$ & $p$
      & $\|e_{N_i}\|$ & $p$ & $\|e_{N_i}\|$ & $p$\\
\midrule
5   & 8.80e-3 & --   & 4.85e-2 & --   & 1.59e-2 & --   & 6.57e-3 & --   \\
10  & 6.80e-3 & 0.37 & 2.19e-2 & 1.15 & 5.66e-3 & 1.49 & 1.15e-3 & 2.51 \\
20  & 3.98e-3 & 0.77 & 1.04e-2 & 1.07 & 1.78e-3 & 1.66 & 2.17e-4 & 2.41 \\
40  & 2.13e-3 & 0.90 & 5.11e-3 & 1.03 & 5.11e-4 & 1.80 & 4.37e-5 & 2.31 \\
80  & 1.10e-3 & 0.95 & 2.53e-3 & 1.02 & 1.38e-4 & 1.89 & 9.51e-6 & 2.20 \\
160 & 5.60e-4 & 0.98 & 1.26e-3 & 1.01 & 3.58e-5 & 1.94 & 2.19e-6 & 2.12 \\
320 & 2.82e-4 & 0.99 & 6.27e-4 & 1.00 & 9.11e-6 & 1.97 & 5.23e-7 & 2.06 \\
\midrule
& \multicolumn{4}{c}{$p = 3$} & \multicolumn{4}{c}{$p = 4$} \\
\cmidrule(lr){2-5} \cmidrule(lr){6-9}
& \multicolumn{2}{c}{GLMQS} & \multicolumn{2}{c}{Bra\'s \cite{Bras2011}} & \multicolumn{2}{c}{GLMQS} & \multicolumn{2}{c}{Bra\'s \cite{Bras2011}}\\
\cmidrule(lr){2-3} \cmidrule(lr){4-5} \cmidrule(lr){6-7}\cmidrule(lr){8-9}
$N_i$ & $\|e_{N_i}\|$ & $p$ & $\|e_{N_i}\|$ & $p$
      & $\|e_{N_i}\|$ & $p$ & $\|e_{N_i}\|$ & $p$\\
\midrule
5   & 6.65e-3 & --   & 1.33e-2 & --   & 2.64e-3 & --   & 1.67e-2 & --   \\
10  & 7.18e-4 & 3.21 & 1.48e-3 & 3.17 & 6.73e-4 & 1.97 & 6.08e-4 & 4.78 \\
20  & 7.08e-5 & 3.34 & 1.25e-4 & 3.57 & 5.66e-5 & 3.57 & 1.62e-5 & 5.22 \\
40  & 5.80e-6 & 3.61 & 8.07e-6 & 3.90 & 2.88e-6 & 4.29 & 1.46e-6 & 3.47 \\
80  & 4.21e-7 & 3.78 & 6.86e-7 & 3.61 & 1.16e-7 & 4.64 & 2.01e-8 & 6.19 \\
160 & 2.85e-8 & 3.88 & 1.00e-7 & 2.77 & 2.46e-9 & 5.56 & 3.42e-9 & 2.56 \\
320 & 1.85e-9 & 3.95 & 1.50e-8 & 2.74 & 3.08e-10 & 3.00 & 9.82e-10 & 1.80 \\
\bottomrule
\end{tabular*}
\end{small}
\end{table}
\begin{figure}[H]
    \centering
    \includegraphics[height=8cm,width=0.85\linewidth]{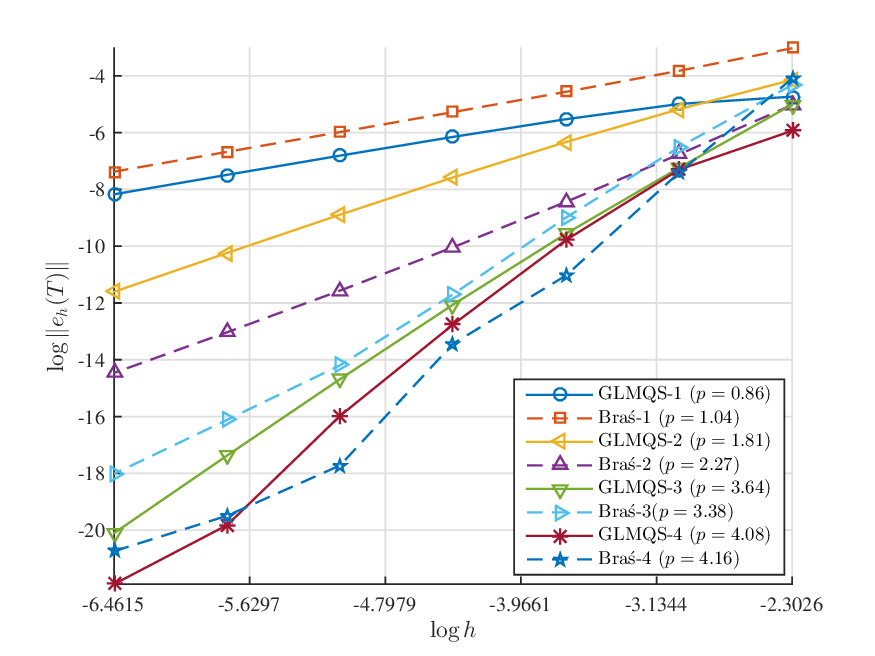}
    \caption{Convergence plot for van der Pol problem}
    \label{fig1}
\end{figure}
\subsection{Burgers' Equation} 
We consider an example of a PDE governed by Burger's equation \cite{Hundsdorfer2003}, which is a one-dimensional quasi-linear parabolic partial differential equation
\begin{equation}\label{Burger}
    \begin{aligned}
        &\frac{\partial u}{\partial t}=-u \frac{\partial u}{\partial x}+d \frac{\partial^2 u}{\partial x^2}, \quad 0<x<L, \quad 0<t<T,
    \end{aligned}
\end{equation}
and the initial condition,
\begin{equation}
    u(x, 0)=\sin (\pi x/L), \quad 0<x<L,
\end{equation}
and boundary conditions,
\begin{equation}
    u(0, t)=u(L, t)=0, \quad t>0.
\end{equation}
Here $d>0$ is the coefficient of viscosity (diffusion coefficient), which can often become very large. Burgers' equation can be used to model a lot of different problems \cite{Hundsdorfer2003, Jackiewicz2009}. For instance, Burgers' equation is another way to look at the Navier-Stokes equations. It has nonlinear parts like unknown functions multiplied by a first derivative and higher-order parts multiplied by a small constant.

For illustration purposes, we choose the final space point $L=1$ and the final time point $T=1$. The spatial discretization is done by taking step size $k=1/(M-1)$, where $M$ is the number of spatial grid points, for $x \in [0,1]$ where $x_0 = 0 \text{ and } L = 1$. We use the central difference approximation to discretize the second-order spatial derivative (diffusion term) on the right-hand side of \eqref{Burger} and the upwind scheme to discretize the first-order derivative (convection term). After the spatial discretization the resulting system of time dependent ODEs is stiff in nature \cite{Hundsdorfer2003}.
\begin{table}[h]
\caption{End‑point $L_2$~‑absolute error norm $\|e_{N_i}\|$ and estimated order $p$ for Burgers’ equation ($T=1$, $d=0.1$)}
\label{table2}
\begin{small}
\begin{tabular*}{\textwidth}{@{\extracolsep\fill}lcccccccc}
\toprule
& \multicolumn{4}{c}{$p = 1$} & \multicolumn{4}{c}{$p = 2$} \\
\cmidrule(lr){2-5} \cmidrule(lr){6-9}
& \multicolumn{2}{c}{GLMQS} & \multicolumn{2}{c}{Bra\'s \cite{Bras2011}} & \multicolumn{2}{c}{GLMQS} & \multicolumn{2}{c}{Bra\'s \cite{Bras2011}}\\
\cmidrule(lr){2-3} \cmidrule(lr){4-5} \cmidrule(lr){6-7}\cmidrule(lr){8-9}
$N_i$ & $\|e_{N_i}\|$ & $p$ & $\|e_{N_i}\|$ & $p$
      & $\|e_{N_i}\|$ & $p$ & $\|e_{N_i}\|$ & $p$\\
\midrule
20   & 1.13e-2 & --   & 3.01e-2 & --   & 5.06e-4 & --   & 1.17e-3 & --   \\
40   & 6.33e-3 & 0.84 & 1.53e-2 & 0.97 & 1.23e-4 & 2.04 & 2.41e-4 & 2.28 \\
80   & 3.35e-3 & 0.92 & 7.73e-3 & 0.99 & 3.00e-5 & 2.03 & 5.22e-5 & 2.21 \\
160  & 1.72e-3 & 0.96 & 3.88e-3 & 0.99 & 7.40e-6 & 2.02 & 1.19e-5 & 2.13 \\
320  & 8.74e-4 & 0.98 & 1.95e-3 & 1.00 & 1.83e-6 & 2.01 & 2.84e-6 & 2.07 \\
640  & 4.40e-4 & 0.99 & 9.74e-4 & 1.00 & 4.57e-7 & 2.01 & 6.91e-7 & 2.04 \\
1280 & 2.21e-4 & 0.99 & 4.87e-4 & 1.00 & 1.14e-7 & 2.00 & 1.70e-7 & 2.02 \\
\midrule
& \multicolumn{4}{c}{$p = 3$} & \multicolumn{4}{c}{$p = 4$} \\
\cmidrule(lr){2-5} \cmidrule(lr){6-9}
& \multicolumn{2}{c}{GLMQS} & \multicolumn{2}{c}{Bra\'s \cite{Bras2011}} & \multicolumn{2}{c}{GLMQS} & \multicolumn{2}{c}{Bra\'s \cite{Bras2011}}\\
\cmidrule(lr){2-3} \cmidrule(lr){4-5} \cmidrule(lr){6-7}\cmidrule(lr){8-9}
$N_i$ & $\|e_{N_i}\|$ & $p$ & $\|e_{N_i}\|$ & $p$
      & $\|e_{N_i}\|$ & $p$ & $\|e_{N_i}\|$ & $p$\\
\midrule
20   & 4.97e-4 & --   & 4.33e-4 & --   & 2.28e-4 & --   & 2.89e-4 & --   \\
40   & 4.00e-5 & 3.64 & 6.03e-5 & 2.84 & 7.63e-6 & 4.90 & 1.15e-5 & 4.65 \\
80   & 2.82e-6 & 3.82 & 8.07e-6 & 2.90 & 3.34e-7 & 4.51 & 2.94e-7 & 5.29 \\
160  & 1.87e-7 & 3.91 & 1.04e-6 & 2.95 & 3.20e-8 & 3.38 & 1.76e-8 & 4.07 \\
320  & 1.21e-8 & 3.95 & 1.33e-7 & 2.98 & 2.45e-9 & 3.71 & 1.66e-9 & 3.41 \\
640  & 7.70e-10 & 3.97 & 1.67e-8 & 2.99 & 1.70e-10 & 3.85 & 1.29e-10 & 3.68 \\
1280 & 5.03e-11 & 3.94 & 2.10e-9 & 2.99 & 1.46e-11 & 3.55 & 1.25e-11 & 3.37 \\
\bottomrule
\end{tabular*}
\end{small}
\end{table}
\begin{figure}[H]
\centering
\includegraphics[height=8cm,width=0.85\linewidth]{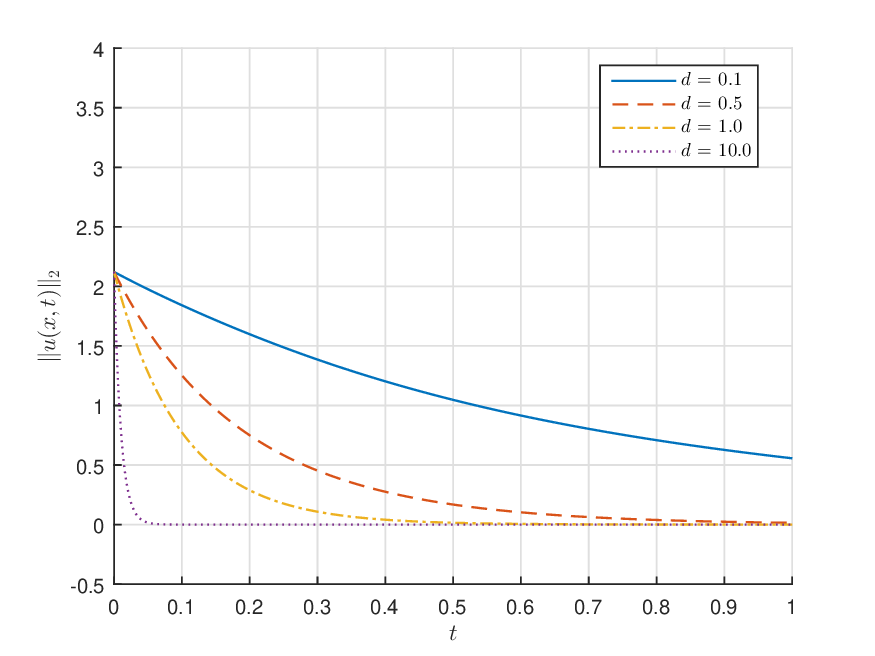}
\caption{Effect of diffusion coefficient using GLMQS-4 for Burgers' equation with $h=1/1280$}
\label{fig3}
\end{figure}
\begin{figure}[H]
\centering
\includegraphics[height=8cm,width=0.85\linewidth]{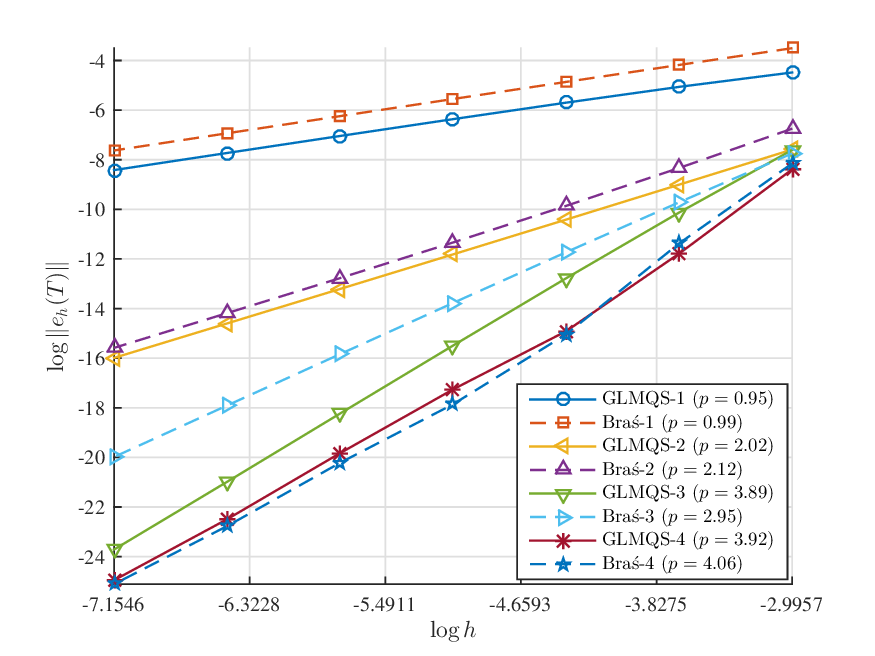}
    \caption{Convergence plots: $\log\|e_h(T)\|$ versus $\log h$ for Burgers' equation with $d=1/10$}
    \label{fig4}
\end{figure}

 \begin{figure}[H]
\centering
\includegraphics[height=8cm,width=0.85\linewidth]{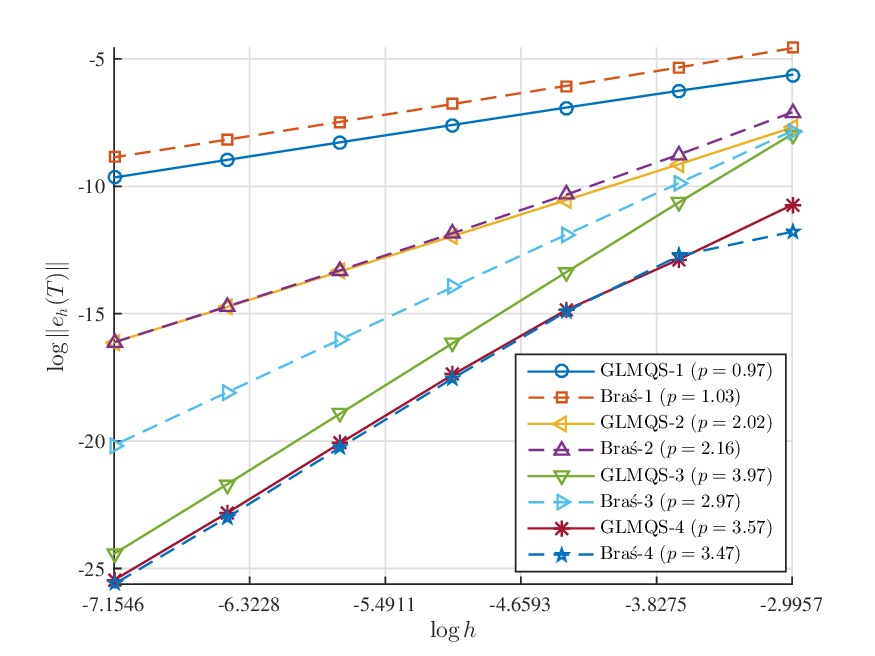}
    \caption{Convergence plots: $\log\|e_h(T)\|$ versus $\log h$ for Burgers' equation with $d=1/2$}
    \label{fig5}
\end{figure}
\begin{figure}[H]
\centering
\begin{subfigure}{0.49\textwidth}
    \centering
    \includegraphics[height=6cm,width=\textwidth]{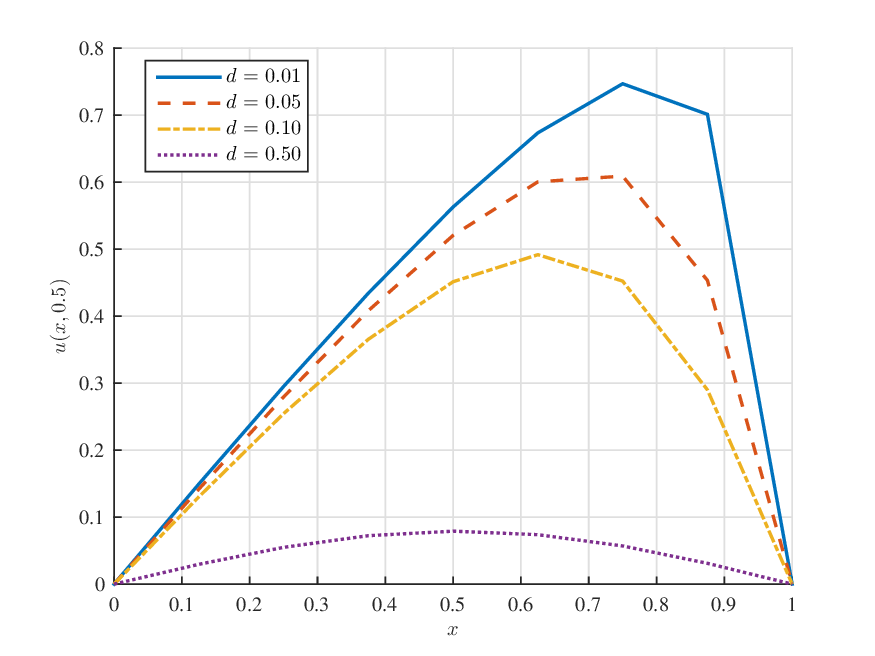}
    \caption{$k=1/9$}
    \label{fig7(a)}
\end{subfigure}
\hfill
\begin{subfigure}{0.49\textwidth}
    \centering
    \includegraphics[height=6cm,width=\textwidth]{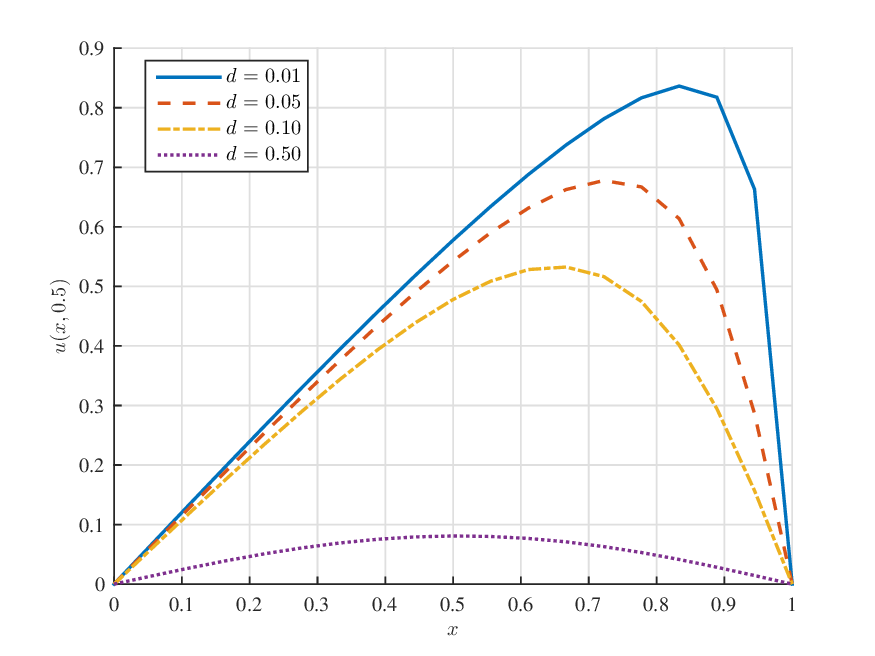}
    \caption{$k=1/19$}
    \label{fig7(b)}
\end{subfigure}
\caption{Numerical solution behaviors with fixed $t=0.5$ and $h=1/640$ for various diffusion coefficients}
\label{fig7}
\end{figure}
\subsection{Gray-Scott Model}
Next, we test our proposed schemes on a coupled PDE known as the Gray-Scott model \cite{Pearson1993}, which is given as follows:
\begin{equation}\label{eq:grayscott}
\begin{aligned}
    \frac{\partial u}{\partial t} &= d_1 \nabla^2 u - u v^2 + F(1-u), \\
    \frac{\partial v}{\partial t} &= d_2 \nabla^2 v + u v^2 - (F+\kappa)v,
\end{aligned}
\end{equation}
where $d_1$ and $d_2$ are the diffusion coefficients corresponding to the concentrations of the chemical species $u$ and $v$ on the square spatial domain $\Omega = [0,L]\times[0,L]$, $L>0$. Here, $F$ is the feed rate and $\kappa$ is the kill rate. 
The spatial domain is discretized on a uniform $M\times M$ grid with spacing $k=\Delta x= \Delta y = L/M,$ which gives a standard five-point stencil for the second-order derivative approximation $u$ and $v$. For $\phi \in \{u,v\}$
\begin{equation}\label{eq:laplacian}
\nabla^2 \phi_{i,j} \approx \frac{\phi_{i+1,j} + \phi_{i-1,j} + \phi_{i,j+1} + \phi_{i,j-1} - 4\phi_{i,j}}{k^2}.
\end{equation}
For the numerical simulation, we choose the final space point $L=1$ and the final time point $T=1$. We use the parameter values 
\begin{equation*}
d_1 = 2\times 10^{-5}, 
\quad d_2 = 1\times 10^{-5}, 
\quad F = 0.04,
\quad \kappa = 0.06.   
\end{equation*}
The boundary conditions are taken as homogeneous Neumann boundary conditions
\begin{equation}
    \frac{\partial u}{\partial n} = \frac{\partial v}{\partial n} = 0 
    \qquad \text{on } \partial\Omega,
\end{equation}
with initial conditions
\begin{equation}
u(x,y,0) = 1, 
\qquad v(x,y,0) = 0.
\end{equation}
Also, consider a small random perturbation of $v$ inside a localized square subregion of the domain, i.e.,
\begin{equation}
v(x,y,0) = \mathcal{O}(10^{-1}) \quad \text{for } (x,y)\in \Omega_c \subset \Omega,
\end{equation}
where $\Omega_c$ denotes a small square centered in $\Omega$. Although \eqref{eq:grayscott} is a well-known PDE for pattern formation, such patterns typically emerge only when the simulations are carried out over very long time intervals, e.g., $T = 1000$ or $T = 10000$. This is not feasible for the proposed schemes, since they are implicit in nature. Therefore, we refrain from demonstrating the pattern-formation phenomena.
 
\begin{table}[h]
\caption{End-point $L_2$~- relative error norm $\|e_{N_i}\|$ and estimated order $p$ for Gray-Scott model}
\label{table3}
\begin{small}
\begin{tabular*}{\textwidth}{@{\extracolsep\fill}lcccccccc}
\toprule
& \multicolumn{4}{c}{$p = 1$} & \multicolumn{4}{c}{$p = 2$} \\
\cmidrule(lr){2-5} \cmidrule(lr){6-9}
& \multicolumn{2}{c}{GLMQS} & \multicolumn{2}{c}{Bra\'s \cite{Bras2011}} & \multicolumn{2}{c}{GLMQS} & \multicolumn{2}{c}{Bra\'s \cite{Bras2011}}\\
\cmidrule(lr){2-3} \cmidrule(lr){4-5} \cmidrule(lr){6-7}\cmidrule(lr){8-9}
$N_i$ & $\|e_{N_i}\|$ & $p$ & $\|e_{N_i}\|$ & $p$
      & $\|e_{N_i}\|$ & $p$ & $\|e_{N_i}\|$ & $p$\\
\midrule
10   & 4.84e-5 & --   & 1.89e-4 & --   & 1.08e-6 & --   & 1.58e-6 & --   \\
20   & 2.99e-5 & 0.69 & 9.33e-5 & 1.02 & 2.64e-7 & 2.02 & 2.98e-7 & 2.41 \\
40   & 1.63e-5 & 0.87 & 4.64e-5 & 1.01 & 6.57e-8 & 2.01 & 6.26e-8 & 2.25 \\
80   & 8.52e-6 & 0.94 & 2.31e-5 & 1.00 & 1.64e-8 & 2.00 & 1.42e-8 & 2.14 \\

\midrule
& \multicolumn{4}{c}{$p = 3$} & \multicolumn{4}{c}{$p = 4$} \\
\cmidrule(lr){2-5} \cmidrule(lr){6-9}
& \multicolumn{2}{c}{GLMQS} & \multicolumn{2}{c}{Bra\'s \cite{Bras2011}} & \multicolumn{2}{c}{GLMQS} & \multicolumn{2}{c}{Bra\'s \cite{Bras2011}}\\
\cmidrule(lr){2-3} \cmidrule(lr){4-5} \cmidrule(lr){6-7}\cmidrule(lr){8-9}
$N_i$ & $\|e_{N_i}\|$ & $p$ & $\|e_{N_i}\|$ & $p$
      & $\|e_{N_i}\|$ & $p$ & $\|e_{N_i}\|$ & $p$\\
\midrule
10   & 1.16e-7 & --   & 1.62e-7 & --   & 5.83e-10 & --   & 1.18e-9 & --   \\
20   & 7.01e-9 & 4.05 & 1.93e-8 & 3.06 & 8.56e-11 & 2.77 & 3.57e-11 & 5.04 \\
40   & 4.30e-10 & 4.02 & 2.37e-9 & 3.03 & 7.96e-12 & 3.43 & 5.72e-12 & 2.64 \\
80   & 3.19e-11 & 3.75 & 2.83e-10 & 3.06 & 6.44e-13 & 3.63 & 4.73e-13 & 3.60 \\

\bottomrule
\end{tabular*}
\end{small}
\end{table}

\begin{figure}[h]
\centering
\includegraphics[height=8cm,width=0.85\linewidth]{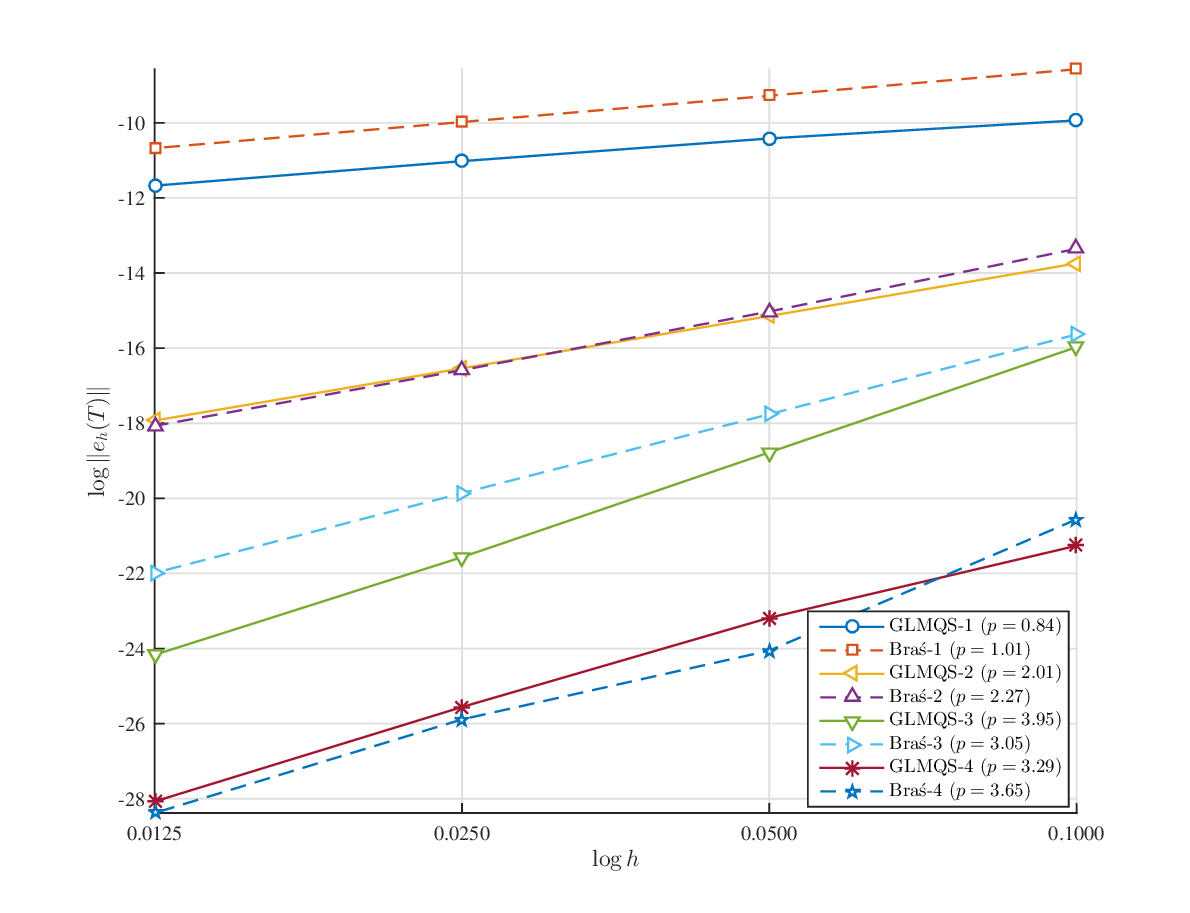}
\caption{Convergence plot for Gray-Scott model}
\label{fig8}
\end{figure}
\section{Results and Discussion}\label{sec6}
The GLMQS schemes constructed in Section~\ref{sec4} were tested on three representative differential systems: the van der Pol oscillator, the Burgers’ equation, and the Gray--Scott model. All simulations were carried out in \texttt{MATLAB} on a 64-bit Windows system (Intel\textsuperscript{\textregistered} Core\texttrademark{} i5, 12\textsuperscript{th} generation, 8~GB RAM). The reference solution in all cases was obtained with \texttt{ode15s}. Performance was evaluated through endpoint error norms, convergence plots, and observed orders, with comparisons to the $p=s$ class introduced in \cite{Bras2011}.

For the van der Pol problem, which serves as a classical stiff ODE benchmark, the convergence plot in Figure~\ref{fig1} and endpoint error norms in Table~\ref{table1} confirm that the proposed GLMQS schemes achieve their theoretical order without order reduction. The observed orders, computed using the estimation formula \eqref{OrderEst.}, closely match theoretical predictions, demonstrating the reliability of the schemes on nonlinear stiff dynamics. 

The results of numerical simulation for Burgers' equation~\eqref{Burger} are presented in the form of the effect of diffusion coefficient in the solution norm using GLMQS-4, as reported in Figure~\ref{fig3}. To demonstrate the convergence, error versus step-size plots in the logarithmic sense are reported for $k=1/9$ with $d=1/10$ (Figure~\ref{fig4}) and $d=1/2$ (Figure~\ref{fig5}). The endpoint error norms are compared in Table~\ref{table2} with those obtained by the class given in \cite{Bras2011}. Further, the effect of different diffusion coefficients on the numerical solution for fixed $t$ using our proposed schemes is depicted in Figure~\ref{fig7} (Figure~\ref{fig7(a)}--\ref{fig7(b)}).

The Gray--Scott reaction--diffusion system provides a challenging test of robustness in PDEs with complex dynamics. Convergence and order are validated and are reported in Table \ref{table3} and Figure \ref{fig8}. The proposed integrators of order $p=1,3,$ produce less error than the $p=s$ class of \cite{Bras2011}, and the method of order $p=2,4,$ is competitive for the choice of parameters made while constructing the example method. A different choice of free parameter, along with a suitable minimization, would lead to a more efficient class with the same or better efficiency. 

Several important inferences can be drawn from the results of our numerical simulations
\begin{itemize}
\item It is evident from the results reported in Tables \ref{table1}, \ref{table2}, and \ref{table3} that the endpoint error norm, obtained for all the problems, shows the convergence of our schemes and is competitive with the previous classes of \cite{Bras2011}. It can also be validated from the convergence plots depicted in the Figures \ref{fig1}, \ref{fig4}, \ref{fig5}, and \ref{fig8}.
\item Our proposed classes do not report any order reduction, and it is evident from the tabulated results.
\item Also, the order estimation given in the Figures \ref{fig1}, \ref{fig4}, \ref{fig5} and \ref{fig8} confirms that our proposed classes are almost free from any kind of order reduction. Moreover, our example methods produces small endpoint error norms than the method given in \cite{Bras2011}.
\item The numerical solution at $t=0.5$ for different diffusion coefficients, as illustrated in Figure \ref{fig7} (\ref{fig7(a)}-\ref{fig7(b)}), preserves the parabolic character of the problem. 
\end{itemize}

\section{Conclusion}\label{sec7}
We have explained the construction of implicit GLMQS and presented four classes of $A$-stable and $L$-stable implicit numerical solvers for solving time-dependent stiff differential equations. The proposed classes of GLM are based on Nordsieck input vectors that satisfy the inherent quadratic stability criteria, along with the minimization of the error constant. Example methods of order $p=1,2,3,$ and $4$ are derived, and their governing coefficient matrices are reported in Section \ref{sec4}. We have also provided three numerical illustrations of time-dependent IVPs, both ODEs and PDEs, to test our proposed classes. The numerical results are reported for the van der Pol problem (ODE) and for both PDEs, the Burgers' equation and the Gray-Scott model. These reported results show that our methods are convergent and achieve the expected order of accuracy. In future articles, we will be focusing on the construction of higher-order methods and the development of implicit-explicit numerical schemes using the present classes of implicit schemes for solving time-dependent partitioned differential equations.\\[10pt]

\noindent
\textbf{Funding.}
The authors have received no funding.\\[6pt]
\textbf{Data availability.}
Not applicable.\\[6pt]
\textbf{Declarations.}
The authors report there are no competing interests to declare.\\[6pt]
\textbf{Acknowledgments.}
The authors are grateful to the editor and the anonymous reviewers for their valuable suggestions, which helped to improve the presentation of this paper.\\


\begin{thebibliography}{00}
    
\bibitem{Hundsdorfer2003} W. Hundsdorfer, J. Verwer, \textit{Numerical Solution of Time-Dependent Advection-Diffusion-Reaction Equations}, Springer, Berlin, Heidelberg, 2003.	

\bibitem{Butcher2016} 
J.C. Butcher, \textit{Numerical Methods for Ordinary Differential Equations}, 3rd ed., John Wiley \& Sons, Chichester, 2016.

\bibitem{Butcher2003} J.C. Butcher, W.M. Wright, The construction of practical general linear methods, \textit{BIT. Numerical Mathematics}. \textbf{43}(4) (2003) 695--721.

 \bibitem{Butcher1993} J.C. Butcher, Diagonally-implicit multi-stage integration methods. \textit{Applied Numerical Mathematics}. \textbf{11}(5) (1993) 347--363.

 \bibitem{Jaust2018} A. Jaust, J. Schütz, General Linear Methods for Time-Dependent PDEs, in: C. Klingenberg, M. Westdickenberg (eds) \textit{Theory, Numerics and Applications of Hyperbolic Problems II. HYP 2016. Springer Proceedings in Mathematics \& Statistics}, \textbf{237} 2018 Springer, Cham. 

\bibitem{Ghahremani2025} Z. Ghahremani, A. Abdi, G. Hojjati, Efficient numerical methods based on general linear methods for Volterra integral equations,\textit{ Applied Numerical Mathematics}, \textbf{216} (2025) 1--16. 

\bibitem{Yu2025} Y. Yu, Solving nonlinear neutral delay integro-differential equations via general linear methods, \textit{Journal of Computational and Applied Mathematics}, \textbf{458} (2025) 116342, 

\bibitem{Izzo2025} G. Izzo, Z. Jackiewicz, Self-starting general linear methods with Runge-Kutta stability. \textit{Journal of Computational Dynamics}, \textbf{12}(1)  (2025) 1--22.  

\bibitem{Jackiewicz2009} Z. Jackiewicz, \textit{General Linear Methods for Ordinary Differential Equations}, John Wiley \& Sons, 2009.

\bibitem{Cardone2011} A. Cardone, Z. Jackiewicz, Explicit Nordsieck methods with quadratic stability, \textit{Numerical Algorithms}. \textbf{60}(1) (2011) 1--25. 

\bibitem{Conte2010} D. Conte, R. D'Ambrosio, Z. Jackiewicz, Two-step Runge-Kutta methods with quadratic stability functions, \textit{Journal of Scientific Computing}. \textbf{44}(2) (2010) 191--218.

\bibitem{Cardone2012} A. Cardone, Z. Jackiewicz, H. Mittelmann, Optimization-based search for Nordsieck methods of high order with quadratic stability polynomials, \textit{Mathematical Modelling and Analysis}. \textbf{17}(3) (2012) 293--308.


\bibitem{Bras2013} M. Bra\'s, A. Cardone, R. D'Ambrosio, Implementation of explicit Nordsieck methods with inherent quadratic stability, \textit{Mathematical Modelling and Analysis}. \textbf{18} (2013) 289--307.

\bibitem{Bras2011} M. Bra\'s, Nordsieck methods with inherent quadratic stability, \textit{Mathematical Modelling and Analysis} \textbf{16}(1) (2011) 82--96.

\bibitem{Bras2014} M. Bra\'s, Z. Jackiewicz, Efficient general linear methods of high order with inherent quadratic stability, \textit{Mathematical Modelling and Analysis}. \textbf{19}(4) (2014) 450--468.

\bibitem{Bras2014a} M. Bra\'s, Z. Jackiewicz, Search for efficient general linear methods for ordinary differential equations, \textit{Journal of Computational and Applied Mathematics}. \textbf{262} (2014) 180--192.

\bibitem{Bras2020} M. Bra\'s, Z. Jackiewicz, A new class of efficient general linear methods for ordinary differential equations. \textit{Applied Numerical Mathematics}. \textbf{151} (2020) 282--300.

\bibitem{Butcher2001} J.C. Butcher, General Linear Methods for Stiff Differential Equations, \textit{BIT Numerical Mathematics}. \textbf{41} (2001) 240–264.

\bibitem{Butcher1997} J.C. Butcher, Z. Jackiewicz, Implementation of diagonally implicit multistage integration methods for ordinary differential equations, \textit{SIAM Journal on Numerical Analysis}. \textbf{34}(6) (1997) 2119--2141.

   
\bibitem{Huang2005} S.J.Y. Huang, \textit{Implementation of General Linear Methods for Stiff Differential Equations}, PhD thesis, The University of Auckland, New Zealand, 2005.

\bibitem{Wright2002} W.M. Wright, Explicit general linear methods with inherent Runge-Kutta stability, \textit{Numerical Algorithms}. \textbf{31}(1-4) (2002) 381--399.

\bibitem{Bras2012} M. Bra\'s,  A. Cardone, Construction of efficient general linear methods for non-stiff differential systems, \textit{Mathematical Modelling and Analysis}. \textbf{17}(2) (2012) 171--189.            

\bibitem{Karaagac2021} B. Karaagac, Numerical treatment of Gray-Scott model with operator splitting method, \textit{Discrete and Continuous Dynamical Systems} - S, \textbf{14(7)} (2021) 2373-2386. doi: 10.3934/dcdss.2020143

\bibitem{Califano2017} G. Califano, G. Izzo, Z. Jackiewicz, Starting procedures for general linear methods, \textit{Applied Numerical Mathematics}. \textbf{120}  (2017) 165--175. 
            
\bibitem{Hairer1996} E. Hairer, G. Wanner, \textit{Solving Ordinary Differential Equations II: Stiff and Differential-Algebraic Problems, Second Revised Edition}, Springer-Verlag, Berlin, Heidelberg, New York.  1996.
         
 \bibitem{Pearson1993} J. E. Pearson, Complex Patterns in a Simple System, \textit{Science} \textbf{261}, (1993) 189--192. 

\end{thebibliography}
\end{document}